\documentclass[12pt]{amsart}

\usepackage{times}
\usepackage{amsmath,amsfonts,amssymb,amsopn,amscd,amsthm}
\usepackage{mathbbol}
\usepackage{mathrsfs}
\usepackage{graphicx,xspace}
\usepackage{epsfig}
\usepackage{enumerate} 
\usepackage{enumitem}
\usepackage{xcolor}
\usepackage[all]{xypic}

\usepackage[T1]{fontenc}
\usepackage[utf8]{inputenc}

\usepackage{hyperref}
\usepackage{psfrag}
\usepackage{bm}
\usepackage{youngtab}
\usepackage{tikz}
\usetikzlibrary{snakes}
\usepackage[all]{xy}
\usepackage{varioref}

\theoremstyle{plain}
\newtheorem{theoreme}{Theorem}[section]
\newtheorem{prop}[theoreme]{Proposition}

\newtheorem{cor}[theoreme]{Corollary}

\theoremstyle{remark}
\newtheorem{definition}[theoreme]{Definition}
\newtheorem{ex}[theoreme]{Example}
\newtheorem{remark}[theoreme]{Remark}
\newtheorem{c-ex}[theoreme]{Counter-example}

\newtheorem*{notation}{Notation}

\date{}

\DeclareUnicodeCharacter{03C7}{\chi}
\DeclareUnicodeCharacter{B2}{^{2}}

\DeclareMathOperator{\ty}{type}
\DeclareMathOperator{\mty}{marked-type}

\DeclareMathOperator{\Res}{Res}

\DeclareMathOperator{\diag}{diag}

\DeclareMathOperator{\character}{char}

\author[O. Tout]{Omar Tout}
\address{Department of Mathematics, Faculty of Sciences III, Lebanese University, Tripoli, Lebanon}
\email{omar-tout@outlook.fr}

\title[On the symmetric Gelfand pair $(\mathcal{H}_n\times \mathcal{H}_{n-1},\diag (\mathcal{H}_{n-1}))$]{On the symmetric Gelfand pair \\$(\mathcal{H}_n\times \mathcal{H}_{n-1},\diag (\mathcal{H}_{n-1}))$}

\begin{document}
\maketitle

\begin{abstract} We show that the $\mathcal{H}_{n-1}$-conjugacy classes of $\mathcal{H}_n,$ where $\mathcal{H}_n$ is the hyperoctahedral group on $2n$ elements, are indexed by marked bipartitions of $n.$ This will lead us to prove that $(\mathcal{H}_n\times \mathcal{H}_{n-1},\diag (\mathcal{H}_{n-1}))$ is a symmetric Gelfand pair and that the induced representation $1_{\diag (\mathcal{H}_{n-1})}^{\mathcal{H}_n\times \mathcal{H}_{n-1}}$ is multiplicity free. 
%The pair $(\mathcal{S}_k\wr \mathcal{S}_n\times \mathcal{S}_k\wr \mathcal{S}_{n-1},\diag \mathcal{S}_k\wr \mathcal{S}_{n-1})$ is a symmetric Gelfand pair for $k=1.$ In this paper we will show that it is also a symmetric Gelfand pair for $k=2$ but fails to be a symmetric Gelfand pair for $k\geq 3$ in general.
\end{abstract}

\section{Introduction}
\iffalse
If $n$ is a positive integer, we will denote by $\mathcal{S}_n$ the symmetric group on $n$ elements. Let $i$ and $k$ be two positive integers. The $k$-tuple $p_{k}(i)$ is the following set of size $k:$ 
$$p_k(i):=\lbrace (i-1)k+1, (i-1)k+2, \ldots , ik\rbrace.$$ 
The usual definition of the wreath product $\mathcal{S}_k \wr \mathcal{S}_n$ can be found in \cite{McDo}. As we showed in \cite{Tout2018}, $\mathcal{S}_k \wr \mathcal{S}_n$ is isomorphic to the subgroup of $\mathcal{S}_{kn}$ formed by permutations that send each set of the form $p_{k}(i),$ $1\leq i\leq n,$ to another set with the same form. Using this fact, we will use the following definition for $\mathcal{S}_k \wr \mathcal{S}_n$ in this paper
$$\mathcal{S}_k \wr \mathcal{S}_n:= \lbrace w \in \mathcal{S}_{kn}; \ \forall \ 1 \leq r \leq n, \ \exists \ 1 \leq r' \leq n \mid w(p_{k}(r))=p_{k}(r')\rbrace.$$
$\mathcal{S}_1 \wr \mathcal{S}_n$ is the symmetric group $\mathcal{S}_n$ and $\mathcal{S}_2 \wr \mathcal{S}_n$ is the hyperoctahedral group which will be denoted $\mathcal{H}_n.$
\fi

If $G$ is a finite group and $K$ is a subgroup of $G,$ we say that the pair $(G,K)$ is a \textit{Gelfand pair} if its associated double-class algebra $\mathbb{C}[K\setminus G/K]$ is commutative, see \cite[chapter VII.1]{McDo}. If $(G,K)$ is a Gelfand pair, the algebra of constant functions over the $K$-double-classes in $G$ has a particular basis, whose elements are called \textit{zonal spherical functions} of the pair $(G,K).$ 

The pair $(G,K)$ is said to be \textit{symmetric} if the action of $G$ on $G/K$ is symmetric, that is if $(x,y)\in \mathcal{O}$ then $(y,x)\in \mathcal{O}$ for all orbits $\mathcal{O}$ of $G$ on $G/K\times G/K.$ It the pair $(G,K)$ is symmetric and Gelfand then it is called a \textit{symmetric Gelfand pair}. It is known, see \cite{ceccherini2010representation}, that $(G,K)$ is a symmetric Gelfand pair if and only if $g^{-1}\in KgK$ for all $g\in G.$

The pair $(G\times G,\diag(G)),$ where $\diag(G)$ is the diagonal subgroup of $G\times G,$ is a symmetric Gelfand pair for any finite group $G,$ see \cite[Section VII.1]{McDo}. Its zonal sphercial functions are the normalised irreducible characters of $G.$ But, if $K$ is a subgroup of $G,$ the pair $(G\times K,\diag(K))$ is a symmetric Gelfand pair if and only if $g^{-1}$ and $g$ are $K$-conjugate for any $g\in G.$\\

If $n$ is a positive integer, we will denote by $\mathcal{S}_n$ the symmetric group on the set $[n]:=\lbrace 1,2,\ldots,n\rbrace.$ The hyperoctahedral group $\mathcal{H}_n$ on the set $[2n]$ is the following subgroup of $\mathcal{S}_{2n}$  
\begin{equation*}
\mathcal{H}_n:=\lbrace \omega\in \mathcal{S}_{2n} \text{ such that for any $i\in [n]$ there exists $j\in [n]$ with } \omega(p_2(i))=p_2(j)\rbrace,
\end{equation*}
where $p_2(i):=\lbrace 2i-1,2i\rbrace$ for any $1\leq i\leq n.$ In addition to $(\mathcal{S}_n\times \mathcal{S}_n,\diag(\mathcal{S}_n)),$ many symmetric Gelfand pairs involving symmetric groups where studied in the litterature. The pair $(\mathcal{S}_{2n}, \mathcal{H}_n)$ is a symmetric Gelfand pair studied in \cite[Section VII.2]{McDo}, \cite{Aker20122465} and \cite{toutejc}. In \cite[Example 1.4.10]{ceccherini2010representation},  the authors show that $(\mathcal{S}_{n},\mathcal{S}_{n-k}\times \mathcal{S}_{k})$ is a symmetric Gelfand pair. 

The symmetric Gelfand pair $(\mathcal{S}_n\times \mathcal{S}_{n-1},\diag(\mathcal{S}_{n-1}))$ was considered by Brender in \cite{brender1976spherical}. An explicit formula for its zonal spherical functions was used by Strahov in \cite{strahov2007generalized} to define the generalized characters of the symmetric group. Combinatorial properties and objects arise from these latters such as a Murnaghan–Nakayama type rule and the definition of generalized Schur functions. The goal of this paper is to prove that $(\mathcal{H}_n\times \mathcal{H}_{n-1},\diag (\mathcal{H}_{n-1}))$ is a symmetric Gelfand pair. This will open the question of finding combinatorial properties of the generalized characters of the hyperoctahedral group. We will not answer this problem here.\\

The paper is organised as follows. In Section \ref{sec:symm_gel_pairs}, the necessary definitions and properties of symmetric Gelfand pairs are given. Especially, we show that if $G$ is a finite group and $K$ is a subgroup of $G$ then $(G\times K,\diag K)$ is a symmetric Gelfand pair if and only if $g^{-1}$ and $g$ are $K$-conjugate for any $g\in G.$ In Section \ref{sec:sphe_func}, we revisit the theory of zonal spherical functions of Gelfand pairs. In Section \ref{sec:The_sym_Gel_pair}, we show that the $\mathcal{H}_{n-1}$-conjugacy classes of $\mathcal{H}_n$ are indexed by marked bipartitions of $n$ which will lead us to a simple proof that $(\mathcal{H}_n\times \mathcal{H}_{n-1},\diag (\mathcal{H}_{n-1}))$ is a symmetric Gelfand pair. In addition, we show that the induced representation $1_{\diag (\mathcal{H}_{n-1})}^{\mathcal{H}_n\times \mathcal{H}_{n-1}}$ is multiplicity free then we give a formula that expresses its character as sum of irreducible characters.

\iffalse
The goal of this paper is to study the pair $(\mathcal{S}_k\wr \mathcal{S}_n\times \mathcal{S}_k\wr \mathcal{S}_{n-1},\diag (\mathcal{S}_k\wr \mathcal{S}_{n-1})).$ When $k=1$ it coincides with the pair $(\mathcal{S}_n\times \mathcal{S}_{n-1},\diag(\mathcal{S}_{n-1}))$ already cited. We will prove that in general if $k\geq 3,$ $(\mathcal{S}_k\wr \mathcal{S}_n\times \mathcal{S}_k\wr \mathcal{S}_{n-1},\diag (\mathcal{S}_k\wr \mathcal{S}_{n-1}))$ is not a symmetric Gelfand pair but  when $k= 2,$ it is. This will open the question of finding the generalized characters along with their combinatorial properties of the hyperoctahedral group. We will not answer this problem in this paper.
\fi
\section{Symmetric Gelfand pairs}\label{sec:symm_gel_pairs}

In this section we suppose that $G$ is a finite group and $K$ is a subgroup of $G.$ We denote by $L(G)$ the algebra of complex-valued functions on $G$ and by $L(K\backslash G/K)$ the subalgebra of $L(G)$ of all the functions $f$ that are constant on the double-classes $KxK$ of $G,$ that is:
$$f\in L(K\backslash G/K) \text{ if and only if } f(kxk^{\prime})=f(x) \text{ for any $x\in G$ and $k,k^{\prime}\in K$}.$$
The product of two functions $f_1$ and $f_2$ in $L(G)$ is the convolution product:
$$(f_1\ast f_2)(g):=\sum_{h\in G}f_1(gh^{-1})f_2(h) \text{ for any $g\in G.$}$$ 

\begin{definition} The pair $(G,K)$ is called a Gelfand pair if the algebra $L(K\backslash G/K)$ is commutative.
\end{definition}

Suppose now that the group $G$ acts on a finite set $X.$ We say that an orbit $\mathcal{O}$ of $G$ on $X\times X$ is symmetric if $(x,y)\in \mathcal{O}$ implies that $(y,x)\in \mathcal{O}.$ In addition, if all orbits of $G$ on $X\times X$ are symmetric we say that the action of $G$ on $X$ is symmetric.
\begin{definition} The pair $(G,K)$ is called symmetric if the action of $G$ on $G/K$ is symmetric. If in addition $(G,K)$ is a Gelfand pair then we say that $(G,K)$ is a symmetric Gelfand pair. 
\end{definition}

\begin{prop}\label{prop_symmetri_pair} 
$(G,K)$ is symmetric if and only if $g^{-1}\in KgK$ for all $g\in G.$
\end{prop}

\begin{proof}
Suppose that $(G,K)$ is symmetric and let $g\in G.$ The couple $(K,g^{-1}K)$ belongs to the orbit of $(gK,K)$ since $(K,g^{-1}K)=g^{-1}(gK,K).$ Thus by symmetrisation $(g^{-1}K,K)$ belongs to the orbit of $(gK,K)$ which means that there exists $x\in G$ such that $g^{-1}K=xgK$ and $K=xK.$ Thus $x\in K$ and $g^{-1}\in KgK.$

Suppose now that $g^{-1}\in KgK$ for all $g\in G.$ Any orbit of $G$ on $G/K\times G/K$ has a representative of the form $(K,xK)$ with $x\in G.$ Take an element $(gK,gxK)$ in the orbit of $(K,xK)=x(x^{-1}K,K).$ Then $(gK,gxK)$ belongs to the orbit of $(x^{-1}K,K).$ But by hypothesis there exists $k_1,k_2\in K$ with $x^{-1}=k_1xk_2$ which means that $(x^{-1}K,K)=(k_1xK,K)=k_1(xK,k_1^{-1}K)=k_1(xK,K)=k_1g^{-1}(gxK,gK).$ Thus $(gxK,gK)$ belongs to the orbit of $(x^{-1}K,K)$ and $(G,K)$ is symmetric.
\end{proof}

\begin{prop}\label{prop_symmetric_Gelfand_pair}
$(G,K)$ is a symmetric Gelfand pair if and only if $g^{-1}\in KgK$ for all $g\in G.$
\end{prop}
\begin{proof} By Proposition \ref{prop_symmetri_pair}, we need only to show that if $g^{-1}\in KgK$ for all $g\in G$ then the algebra $L(K\backslash G/K)$ is commutative. For any two functions $f_1,f_2\in L(K\backslash G/K)$ and for any $x\in G$ we have:
\begin{eqnarray*}
(f_1\ast f_2)(x)=\sum_{h\in G}f_1(xh^{-1})f_2(h)&=&\sum_{k\in G,\atop{ k=xh^{-1}}}f_2(k^{-1}x)f_1(k)\\
&=&\sum_{k\in G}f_2(x^{-1}k)f_1(k^{-1}) \,\,\,\,\,\,\,\,\,\,\,\,\,\,\,\,\, (f_1,f_2\in L(K\backslash G/K))\\
&=&(f_2\ast f_1)(x^{-1})\,\,\,\,\,\,\,\,\,\,\,\,\,\,\,\,\,\,\,\,\,\,\, (\text{ definition of convolution})\\
&=&(f_2\ast f_1)(x) \,\,\,\,\,\,\,\,\,\,\,\,\,\,\,\,\,\,\,\,\,\,\,\,\,\,\,\,\,\,\,\,\,\,\,\, (f_2\ast f_1\in L(K\backslash G/K)).
\end{eqnarray*}
\end{proof}

If $K$ is a subgroup of $G$ we will denote by $\diag (K)$ the diagonal subgroup of $G\times K,$ that is the subgroup of $G\times K$ formed of all the elements $(k,k)$ with $k\in K:$
$$\diag (K):=\lbrace (k,k)\in G\times K \rbrace.$$

\begin{definition} Let $K$ be a subgroup of $G.$ Two elements $x,y\in G$ are said to be $K$-conjugate if there exists $k\in K$ such that $x=kyk^{-1}.$
\end{definition}

\begin{prop} \label{prop_couple_sym_Gel}
$(G\times K,\diag (K))$ is a symmetric Gelfand pair if and only if $g^{-1}$ and $g$ are $K$-conjugate for any $g\in G.$
\end{prop}
\begin{proof}
By Proposition \ref{prop_symmetric_Gelfand_pair}, if $(G\times K,\diag (K))$ is a symmetric Gelfand pair then for any $(g,k)\in G\times K,$ $(g,k)^{-1}$ belongs to the $\diag (K)$-double-coset of $(g,k).$ In particular, if $g\in G$ then $(g^{-1},1)$ belongs to the $\diag (K)$-double-coset of $(g,1),$ that is there exist $k_1,k_2\in K$ such that $g^{-1}=k_1gk_2$ and $1=k_1k_2.$ Conversely, if $(g,k)\in G\times K$ then by hypothesis there exists $x\in K$ such that $(gk^{-1})^{-1}=xgk^{-1}x^{-1},$ which implies that $(g,k)^{-1}=(k^{-1}x,k^{-1}x)(g,k)(k^{-1}x^{-1},k^{-1}x^{-1})$ is in the $\diag (K)$-double-coset of $(g,k).$ Thus $(G\times K,\diag (K))$ is a symmetric Gelfand pair again by Proposition \ref{prop_symmetric_Gelfand_pair}.
\end{proof}

\begin{cor}
$(G\times G,\diag (G))$ is a symmetric Gelfand pair for any finite group $G.$
\end{cor}

\begin{cor}\label{cor_symm}
If $(G\times K,\diag (K))$ is a symmetric Gelfand pair then $(G,K)$ is a symmetric Gelfand pair.
\end{cor}
\begin{proof}
By Proposition \ref{prop_couple_sym_Gel}, if $(G\times K,\diag (K))$ is a symmetric Gelfand pair then $g^{-1}$ and $g$ are $K$-conjugate for any $g\in G$ which implies that $g^{-1}\in KgK$ and $(G,K)$ is a symmetric Gelfand pair by Proposition \ref{prop_symmetric_Gelfand_pair}.
\end{proof}
\begin{ex} For any positive integer $n,$ the pair $(\mathcal{S}_n\times \mathcal{S}_{n-1},\diag (\mathcal{S}_{n-1}))$ is a symmetric Gelfand pair as it is shown in \cite[Theorem 3.2.1]{ceccherini2010representation} and in the paper \cite{strahov2007generalized} of Strahov. This implies, using Corollary \ref{cor_symm}, that in particular the pair $(\mathcal{S}_n,\mathcal{S}_{n-1})$ is a symmetric Gelfand pair. This latter fact can be proven easily using Proposition \ref{prop_symmetri_pair} and the fact that $g^{-1}=g^{-1}gg^{-1}$ for any $g\in \mathcal{S}_{n-1}$ and 
\begin{equation*}
g^{-1}=g^{-1}(n\, g(n))gg^{-1}(n\, g(n)) \text{ for any $g\in \mathcal{S}_n\setminus \mathcal{S}_{n-1}$ with $g^{-1}(n\, g(n))$ being in $\mathcal{S}_{n-1}$} . 
\end{equation*}
\end{ex}
The inverse of the above Corollary \ref{cor_symm} is not always true. This means that if $(G,K)$ is a symmetric Gelfand pair then $(G\times K,\diag (K))$ is not necessarily a symmetric Gelfand pair. To prove this we consider the following example involving the hyperocahedral group $\mathcal{H}_n.$ \iffalse If $i$ is a positive integer we will denote by $p_2(i)$ the set containing the elements $2i-1$ and $2i.$ The hyperoctahedral group $\mathcal{H}_n$ on the set $[2n]$ is the following subgroup of $\mathcal{S}_{2n}$  
\begin{equation*}
\mathcal{H}_n:=\lbrace \omega\in \mathcal{S}_{2n} \text{ such that for any $i\in [n]$ there exists $j\in [n]$ with } \omega(p_2(i))=p_2(j)\rbrace.
\end{equation*}
\fi
\begin{ex}
It is well known that the pair $(\mathcal{S}_{2n},\mathcal{H}_n)$ is a symmetric Gelfand pair, see \cite{McDo}. We used the mathematical software Sagemath \cite{SAGE} to check that for $\omega=(1,6,4,5,2,3,8,7)\in \mathcal{S}_8,$ $\omega^{-1}\neq \sigma \omega \sigma^{-1}$ for any $\sigma\in \mathcal{H}_4.$ This implies that the pair $(\mathcal{S}_8\times \mathcal{H}_4,\diag (\mathcal{H}_4))$ is not a symmetric Gelfand pair by Proposition \ref{prop_couple_sym_Gel}. In fact $n=4$ is the least integer for which the pair $(\mathcal{S}_{2n}\times \mathcal{H}_n,\diag (\mathcal{H}_n))$ is not a symmetric Gelfand pair.
\end{ex}

\section{Zonal spherical functions and generalised characters}\label{sec:sphe_func}

Let $\hat{G}$ be a fixed set of irreducible pairwise inequivalent representations of $G.$ A representation of $G$ is said to be multiplicity free if each irreducible representation of $\hat{G}$ appears at most once in its decomposition as sum of irreducible representations. Consider now the algebra $\mathbb{C}[G/K]$ spanned by the set of left cosets of $G$ and let $G$ acts on it as follows:
$$g(k_iK)=(gk_i)K \text{ for every $g\in G$ and every representative $k_i$ of $G/K$}.$$
This representation of $G$ is usually denoted $1_K^G$ and called the induced representation of $K$ on $G.$ For a proof of the following theorem, the reader is invited to check \cite[$(1.1)$ page $389$]{McDo}.

\begin{theoreme}\label{th} The following two conditions are equivalent:
\begin{enumerate}
\item[1-] $(G,K)$ is a Gelfand pair.
\item[2-] the induced representation $1_K^G$ is multiplicity free.
\end{enumerate}
\end{theoreme}

Suppose that $(G,K)$ is a Gelfand pair and that :
$$1_K^G=\bigoplus_{i=1}^{s}X_i,$$
where $X_i$ are irreducible representations of $G.$ We define the functions $\omega_i:G\rightarrow \mathbb{C},$ using the irreducible characters $\mathcal{X}_i,$ as follows :
\begin{equation}
\omega_i(x)=\frac{1}{|K|}\sum_{k\in K}\mathcal{X}_i(x^{-1}k) \text{ for every $x\in G.$} 
\end{equation} 
The functions $(\omega_i)_{1\leq i\leq s}$ are called the \textit{zonal spherical functions} of the pair $(G,K).$ They have a long list of important properties, see \cite[page 389]{McDo}. They form an orthogonal basis for $L(K\backslash G/K)$ and they satisfy the following equality :
\begin{equation}\label{prop_fon_zon}
\omega_i(x)\omega_i(y)=\frac{1}{|K|}\sum_{k\in K}\omega_i(xky),
\end{equation}
for every $1\leq i\leq s$ and for every $x,y\in G.$\\

Let us denote by $C(G,K)$ the subalgebra of $L(G)$ of functions which are constant on the $K$-conjugacy classes, that is:
\begin{equation*}
C(G,K):=\lbrace f\in L(G) \text{ such that $f(x)=f(y)$ for any two $K$-conjugate elements $x,y\in G$}\rbrace.
\end{equation*}
Recall that any irreducible representation of $G\times K$ is of the form $X\times Y$ where $X$ and $Y$ are irreducible representations of $G$ and $K$ respectively. Suppose that $(G\times K,\diag (K))$ is a symmetric Gelfand pair with 
$$1_{\diag (K)}^{G\times K}=\bigoplus_{1\leq i\leq r \atop{1\leq j\leq s}} X_i\times Y_j, $$
where $X_i$ and $Y_j$ are irreducible representations of $G$ and $K$ respectively for any $i$ and $j.$
The zonal spherical functions of the pair $(G\times K,\diag (K))$ are then given by:
\begin{eqnarray*}
\omega_{i,j}(x,y)&=&\frac{1}{|K|}\sum_{h\in K}\mathcal{X}_i(x^{-1}h)\mathcal{Y}_j(y^{-1}h)
\end{eqnarray*} 
for any $(x,y)\in G\times K.$ In \cite[Lemma 2.1.1]{ceccherini2010representation}, the authors showed that the map $$\Phi:L(\diag(K) \setminus G\times K / \diag(K))\rightarrow C(G,K)$$ defined by:
\begin{equation*}
[\Phi(F)](g)=|K|F(g,1_G)
\end{equation*}
is a linear isomorphism of algebras and that 
\begin{equation*}
||\Phi(F)||_{L(G)}^2=|K| ||F||_{L(G\times K)}^2.
\end{equation*}
The image of the functions $\frac{\mathcal{Y}_j(1)}{|K|}\omega_{i,j}$ by $\Phi,$ which shall be denoted $\varpi_{i,j},$ form an orthogonal basis for the algebra $C(G,K).$ They will be called the $K$-generalised characters of the group $G$ and they are explicitly defined by  
\begin{equation}\label{Def_Gen_char}
\varpi_{i,j}(g)=\mathcal{Y}_j(1)\omega_{i,j}(g,1) \text{ for any $g\in G.$}
\end{equation}

The $\mathcal{S}_{n-1}$-generalised characters of the symmetric group $\mathcal{S}_n$ appeared for the first time in the paper \cite{strahov2007generalized} of Strahov. They were introduced using an explicit formula for the zonal spherical functions of the symmetric Gelfand pair $(\mathcal{S}_n\times \mathcal{S}_{n-1},\diag(\mathcal{S}_{n-1})).$ Many interesting combinatorial properties and formulas can be extended from characters to generalised characters of the symmetric group. A Murnaghan–Nakayama type rule, a Frobenius type formula, an analogue of the Jacobi-Trudi formula and of the determinantal formula for the generalised characters of the symmetric group and the generalised Schur functions can be found in \cite{strahov2007generalized}.

\section{The symmetric Gelfand pair $(\mathcal{H}_n\times \mathcal{H}_{n-1},\diag (\mathcal{H}_{n-1}))$}\label{sec:The_sym_Gel_pair}

In this section we will show that the $\mathcal{H}_{n-1}$-conjugacy classes of $\mathcal{H}_n$ are indexed by marked bipartitions of $n.$ In particular, we will show that for any element $\omega\in \mathcal{H}_n,$ $\omega$ and its inverse $\omega^{-1}$ belong to the same $\mathcal{H}_{n-1}$-conjugacy class of $\mathcal{H}_n.$ This implies using Proposition \ref{prop_couple_sym_Gel} that $(\mathcal{H}_n\times \mathcal{H}_{n-1},\diag (\mathcal{H}_{n-1}))$ is a symmetric Gelfand pair.

\subsection{Marked partitions} A \textit{partition} $\lambda$ is a weakly decreasing list of positive integers $(\lambda_1,\ldots,\lambda_l)$ where $\lambda_1\geq \lambda_2\geq\ldots \geq\lambda_l\geq 1.$ The $\lambda_i$ are called the \textit{parts} of $\lambda$; the \textit{size} of $\lambda$, denoted by $|\lambda|$, is the sum of all of its parts. If $|\lambda|=n$, we say that $\lambda$ is a partition of $n.$ 
%and we write $\lambda\vdash n$. The number of parts of $\lambda$ is denoted by $l(\lambda)$. We will also use the exponential notation $\lambda=(1^{m_1(\lambda)},2^{m_2(\lambda)},3^{m_3(\lambda)},\ldots),$ where $m_i(\lambda)$ is the number of parts equal to $i$ in the partition $\lambda.$ In case there is no confusion, we will omit $\lambda$ from $m_i(\lambda)$ to simplify our notation. If $\lambda=(1^{m_1(\lambda)},2^{m_2(\lambda)},3^{m_3(\lambda)},\ldots,n^{m_n(\lambda)})$ is a partition of $n$ then $\sum_{i=1}^n im_i(\lambda)=n.$ We will dismiss $i^{m_i(\lambda)}$ from $\lambda$ when $m_i(\lambda)=0,$ for example, we will write $\lambda=(1^2,3,6^2)$ instead of $\lambda=(1^2,2^0,3,4^0,5^0,6^2,7^0).$ If $\lambda$ and $\delta$ are two partitions, we define the \textit{union} $\lambda \cup \delta$ and subtraction $\lambda \setminus \delta$ (if exists) as the following partitions:
%$$\lambda \cup \delta=(1^{m_1(\lambda)+m_1(\delta)},2^{m_2(\lambda)+m_2(\delta)},3^{m_3(\lambda)+m_3(\delta)},\ldots).$$
%$$\lambda \setminus \delta=(1^{m_1(\lambda)-m_1(\delta)},2^{m_2(\lambda)-m_2(\delta)},3^{m_3(\lambda)-m_3(\delta)},\ldots) \text{ if $m_i(\lambda)\geq m_i(\delta)$ for any $i.$ }$$

In this paper we will face problems that affect a particular part $\lambda_j$ of a partition $\lambda.$ In this case the part $\lambda_j$ will be called a \textit{marked part} of $\lambda.$ If $i$ is a marked part of a partition $\lambda$ then we will call $\lambda$ a marked partition at $i$ and we will write ${\hat{\lambda}}^i$ to designate it. The set of all partitions of $n$ will be denoted $\mathcal{P}_n$ while $\mathcal{MP}_n$ will denote the set of all marked partitions of $n.$ $\mathcal{P}_0$ will be considered to be the empty set $\emptyset.$

\iffalse
By a \emph{ family of partitions } we will always mean a family of partitions $\Lambda=(\Lambda(\lambda))_\lambda$ indexed by the partitions $\lambda$ of $k$ that satisfies:
$$|\Lambda|:=\sum_{\lambda\vdash k}|\Lambda(\lambda)|=n.$$

A family of partitions $\Lambda=(\Lambda(\lambda))_\lambda$ is marked if there exists $\rho\in \mathcal{P}_k$ and a positive integer $i$ such that the partition $\Lambda(\rho)$ is marked at $i.$ In such case we will write $\widehat{\Lambda}^{\rho,i}.$
\fi

A \textit{bipartition} of $n$ is a couple $(\lambda_1,\lambda_2)\in \mathcal{P}_k\times \mathcal{P}_{n-k}$ for some $k\in [n].$ A \textit{marked bipartition} of $n$ is a bipartition of $n$ with $\lambda_1$ or $\lambda_2$ marked. We will write $\widehat{(\lambda_1,\lambda_2)}^{i,r}$ with $i\in \lbrace 1,2\rbrace$ to say that $(\lambda_1,\lambda_2)$ is a marked bipartition where $r$ is the marked part in the partition $\lambda_i.$

\subsection{Conjugacy classes of the hyperoctahedral group}\label{sec_hyp}

We recall that if $i$ is a positive integer then $p_2(i)$ is the set containing the elements $2i-1$ and $2i.$ If $a\in p_2(i)$, we shall denote by $\overline{a}$ the element of the set $p_2(i)\setminus \lbrace a\rbrace$ and thus we have $\overline{\overline{a}}=a.$ %The hyperoctahedral group $\mathcal{H}_n$ on the set $[2n]$ is the following subgroup of $\mathcal{S}_{2n}$  
%\begin{equation*}
%\mathcal{H}_n:=\lbrace \omega\in \mathcal{S}_{2n} \text{ such that for any $i\in [n]$ there exists $j\in [n]$ with } \omega(p_2(i))=p_2(j)\rbrace.
%\end{equation*}

Let us study the cycle decomposition of a permutation $\omega\in \mathcal{H}_n.$ If $\mathcal{C}=(a_1,\cdots ,a_{l})$ is a cycle of $\omega,$ we distinguish the following two cases:
\begin{enumerate}
\item $\overline{a_1}$ appears in the cycle $\mathcal{C},$ for example $a_j=\overline{a_1}.$ Since $\omega\in\mathcal{H}_n$ and $\omega(a_1)=a_2,$ we have $\omega(\overline{a_1})=\overline{a_2}=\omega(a_j).$ Likewise, since $\omega(a_{j-1})=\overline{a_1},$ we have $\omega(\overline{a_{j-1}})=a_1$ which means that $a_{l}=\overline{a_{j-1}}.$ Therefore,
$$\mathcal{C}=(a_1,\cdots a_{j-1},\overline{a_1},\cdots,\overline{a_{j-1}})$$
and $l=2(j-1)$ is even. We will denote such a cycle by $(\mathcal{O},\overline{\mathcal{O}}).$
\item $\overline{a_1}$ does not appear in the cycle $\mathcal{C}.$ Take the cycle $\overline{\mathcal{C}}$ which contains $\overline{a_1}.$ Since $\omega(a_1)=a_2$ and $\omega\in \mathcal{H}_n$, we have $\omega(\overline{a_1})=\overline{a_2}$ and so on. That means that the cycle $\overline{\mathcal{C}}$ has the following form,
$$\overline{\mathcal{C}}=(\overline{a_1},\overline{a_2},\cdots ,\overline{a_{l}})$$
and that $\mathcal{C}$ and $\overline{\mathcal{C}}$ appear in the cycle decomposition of $\omega.$
\end{enumerate} 

Suppose now that the cycle decomposition of a permutation $\omega$ of $\mathcal{H}_n$ is as follows:
\begin{equation}\label{cyc_decom_H_n}
\omega=\mathcal{C}_1\overline{\mathcal{C}_1}\mathcal{C}_2\overline{\mathcal{C}_2}\cdots \mathcal{C}_k\overline{\mathcal{C}_k}(\mathcal{O}^1,\overline{\mathcal{O}^1})(\mathcal{O}^2,\overline{\mathcal{O}^2})\cdots (\mathcal{O}^l,\overline{\mathcal{O}^l}),
\end{equation}
where the cycles $\mathcal{C}_i$ (resp. $(\mathcal{O}^j,\overline{\mathcal{O}^j})$) are written decreasingly according to their sizes. Define $\ty(\omega)$ to be the bipartition $(\lambda_1,\lambda_2)$ of $n$ where $\lambda_1$ (resp. $\lambda_2$) is the partition formed by the lengths of $\mathcal{C}_i$ (resp. $\mathcal{O}^i$).

\begin{ex}\label{main_ex} Consider the following permutation 
$$\omega=\begin{pmatrix}
1&2&&3&4&&5&6&&7&8&&9&10&&11&12&&13&14&&15&16\\
14&13&&1&2&&16&15&&7&8&&12&11&&10&9&&4&3&&5&6
\end{pmatrix}\in \mathcal{H}_{8}.$$ Its decomposition into product of disjoint cycles is as follows:
$$\omega=(1,14,3)(2,13,4)(7)(8)(9,12)(10,11)(5,16,6,15).$$
Thus we have $\ty(\omega)=((3,2,1),(2)).$
\end{ex}

It is well known that the conjugacy classes of the hyperoctahedral group $\mathcal{H}_n$ are indexed by bipartitions of $n$ and that two permutations in $\mathcal{H}_n$ are conjugate if and only if they have the same type, see \cite[Section 2]{stembridge1992projective}. 

\subsection{$\mathcal{H}_{n-1}$-conjugacy classes}
Suppose $\omega\in \mathcal{H}_n$ is written as product of disjoint cycles as in (\ref{cyc_decom_H_n}). We define \textit{marked-type} of $\omega$ to be the marked bipartition
\begin{enumerate}
\item[$\bullet$] $(\widehat{\lambda_1}^i,\lambda_2)$ if one of the elements $2n-1$ and $2n$ belongs to a certain $\mathcal{C}_j$ of length $i$ and the other belongs to $\overline{\mathcal{C}_j}.$
\item[$\bullet$] $(\lambda_1,\widehat{\lambda_2}^i)$ if the elements $2n-1$ and $2n$ belongs to a certain $(\mathcal{O}^j,\overline{\mathcal{O}^j})$ and the number of elements of $\mathcal{O}^j$ is $i.$ 
\end{enumerate}

\begin{ex} The $\mty$ of the permutation $\omega\in \mathcal{H}_8$ in Example \ref{main_ex} is $((3,2,1),\widehat{(2)}^2).$
\end{ex}

\begin{notation} If $c=(a_1,a_2,\ldots, a_l)$ is a cycle and $\alpha$ is a permutation then $\alpha(c)$ will denote the cycle $(\alpha(a_1),\alpha(a_2),\ldots, \alpha(a_l)).$
\end{notation}

\begin{prop}\label{prop_conj_H_n-1}
Two permutations $\omega_1,\omega_2\in \mathcal{H}_n$ are in the same $\mathcal{H}_{n-1}$-conjugacy class if and only if $\mty(\omega_1)=\mty(\omega_2).$ 
\end{prop}
\begin{proof}
If $\omega_1,\omega_2\in \mathcal{H}_n$ are in the same $\mathcal{H}_{n-1}$-conjugacy class then there exists $\alpha\in \mathcal{H}_{n-1}$ such that $\omega_1=\alpha\omega_2\alpha^{-1}.$ There are two cases to consider. Suppose that $\omega_2$ is written as a product of disjoint cycles as in (\ref{cyc_decom_H_n}):
\begin{enumerate}
\item[1-] if there exists $a\in p_2(n)$ and $1\leq j\leq k$ such that $\mathcal{C}_j=(a,\ldots)$ with length $r$ then $\mty(\omega_2)=\widehat{\ty(\omega_2)}^{1,i}.$ The decomposition of $\omega_1$ into product of disjoint cycles is
$$\alpha(\mathcal{C}_1)\alpha(\overline{\mathcal{C}_1})\alpha(\mathcal{C}_2)\alpha(\overline{\mathcal{C}_2})\cdots \alpha(\mathcal{C}_k)\alpha(\overline{\mathcal{C}_k})\alpha((\mathcal{O}^1,\overline{\mathcal{O}^1}))\alpha((\mathcal{O}^2,\overline{\mathcal{O}^2}))\cdots \alpha((\mathcal{O}^l,\overline{\mathcal{O}^l})).$$
Since $\alpha\in \mathcal{H}_n,$ the cycles do not change their lengths after applying $\alpha$ and $\alpha(a)=a$ which shows that $\mty(\omega_1)=\mty(\omega_2)=\widehat{\ty(\omega_2)}^{1,i}.$
\item[2-] if $2n-1$ and $2n$ both appear in a cycle $(\mathcal{O}^l,\overline{\mathcal{O}^l})$ of length $2r$ then the same reasoning as in item $1$ will show that $\mty(\omega_1)=\mty(\omega_2)=\widehat{\ty(\omega_2)}^{2,r}.$
\end{enumerate}
On the other direction, if $\mty(\omega_1)=\mty(\omega_2),$ we also distinguish the following two cases:
\begin{enumerate}
\item[1-] the first partition is disctinguished: the cycle decompositions of $\omega_1$ and $\omega_2$ can be written as follows: 
\begin{equation*}
\omega_1=\mathcal{C}_1\overline{\mathcal{C}_1}\mathcal{C}_2\overline{\mathcal{C}_2}\cdots \mathcal{C}_k\overline{\mathcal{C}_k}(\mathcal{O}^1,\overline{\mathcal{O}^1})(\mathcal{O}^2,\overline{\mathcal{O}^2})\cdots (\mathcal{O}^l,\overline{\mathcal{O}^l}),
\end{equation*}
\begin{equation*}
\omega_2=\mathcal{L}_1\overline{\mathcal{L}_1}\mathcal{L}_2\overline{\mathcal{L}_2}\cdots \mathcal{L}_k\overline{\mathcal{L}_k}(\mathcal{K}^1,\overline{\mathcal{K}^1})(\mathcal{K}^2,\overline{\mathcal{K}^2})\cdots (\mathcal{K}^l,\overline{\mathcal{K}^l}),
\end{equation*}
where the length of each $\mathcal{C}_i$ (resp. $\mathcal{O}^j$) equals the length of $\mathcal{L}_i$ (resp. $\mathcal{K}^j$) and the integer $2n-1$ appears in a cycle $\mathcal{C}_r$ (resp. $\mathcal{L}_r$) while its complement $2n$ appears in $\overline{\mathcal{C}_r}$ (resp. $\overline{\mathcal{L}_r}$). If for each $1\leq i\leq k,$ the integer $a_i$ appears first in the cycle $\mathcal{C}_i$ (resp. $\mathcal{L}_i$) then let $\overline{a_i}$ appears first in the writing of the cycle $\overline{\mathcal{C}_i}$ (resp. $\overline{\mathcal{L}_i}$). It would be clear then that the permutation $\alpha$ that cyclically takes the cycles $\mathcal{C}_i$ to $\mathcal{L}_i,$ $\overline{\mathcal{C}_i}$ to $\overline{\mathcal{L}_i}$ and $(\mathcal{O}^j,\overline{\mathcal{O}^j})$ to $(\mathcal{K}^j,\overline{\mathcal{K}^j})$ belongs to $\mathcal{H}_{n-1}$ and $\omega_2=\alpha\omega_1\alpha^{-1}.$
\item[2-] the second partition is disctinguished: reason in the same way as in item $1$ except that in this case both $2n-1$ and $2n$ appear in a cycle $(\mathcal{O}^j,\overline{\mathcal{O}^j})$ (resp. $(\mathcal{K}^j,\overline{\mathcal{K}^j})$).
\end{enumerate}
\end{proof}

\begin{cor}\label{cor:sym_Gelf_pair} The pair $(\mathcal{H}_n\times\mathcal{H}_{n-1}, \diag (\mathcal{H}_{n-1}) )$ is a symmetric Gelfand pair.
\end{cor}
\begin{proof}
It would be clear from the cycle decomposition that $\mty(\omega)=\mty(\omega^{-1})$ for any $\omega\in \mathcal{H}_n$ which implies that $\omega$ and $\omega^{-1}$ are $\mathcal{H}_{n-1}$-conjugate by Proposition \ref{prop_conj_H_n-1}. The result then follows from Proposition \ref{prop_couple_sym_Gel}.
\end{proof}

\begin{remark} The fact that $(\mathcal{S}_n\times\mathcal{S}_{n-1}, \diag (\mathcal{S}_{n-1}) )$ and $(\mathcal{H}_n\times\mathcal{H}_{n-1}, \diag (\mathcal{H}_{n-1}) )$ are both symmetric Gelfand pairs may suggest that in general $(\mathcal{S}_k\wr \mathcal{S}_n\times \mathcal{S}_k\wr \mathcal{S}_{n-1},\diag \mathcal{S}_k\wr \mathcal{S}_{n-1})$ is a symmetric Gelfand pair for any $k\in \mathbb{N}^*,$ where
$$\mathcal{S}_k \wr \mathcal{S}_n \footnote{ For the wreath product of symmetric groups, we use the definition given in \cite{Tout2018}}:= \lbrace w \in \mathcal{S}_{kn}; \ \forall \ 1 \leq r \leq n, \ \exists \ 1 \leq r' \leq n \mid w(p_{k}(r))=p_{k}(r')\rbrace$$
with $$p_k(i):=\lbrace (i-1)k+1, (i-1)k+2, \ldots , ik\rbrace.$$ This is not true. For example, if $k=3,$ $\omega=(1)(2)(3)(456)\in \mathcal{S}_3\wr \mathcal{S}_2$ then one can easily check that $\omega\neq \alpha\omega^{-1}\alpha^{-1}$ for any $\alpha\in \mathcal{S}_3\wr \mathcal{S}_1=\mathcal{S}_3$ which shows that $\omega$ and its inverse are not $\mathcal{S}_3\wr \mathcal{S}_1$-conjugate. This implies that the pair $(\mathcal{S}_3\wr \mathcal{S}_2\times \mathcal{S}_3\wr \mathcal{S}_{1},\diag (\mathcal{S}_3\wr \mathcal{S}_{1}))$ is not a symmetric Gelfand pair by \ref{prop_couple_sym_Gel}.
\end{remark}

\subsection{Zonal spherical functions of the pair $(\mathcal{H}_n\times\mathcal{H}_{n-1}, \diag (\mathcal{H}_{n-1}) )$}
Any partition $\lambda=(\lambda_1,\ldots,\lambda_l)$ of $n$ can be represented by a Young diagram. This is an array of $n$ squares having $l$ left-justified rows with row $i$ containing $\lambda_i$ squares for $1\leq i\leq l.$ For example, the following is the Young diagram of the partition $\lambda=(5,3,3,1,1)$ of $13$
$$\young({}{}{}{}{}{},{}{}{}{},{}{}{}{},{}{},{}{})$$
An \textit{exterior corner} of a Young diagram $\mathcal{Y}$ having $n$ squares is an extremity of a row where a new square can be added to obtain a new Young diagram with $n+1$ squares. Below we mark by bullets the four exterior corners of the Young diagram of $\lambda=(5,3,3,1,1)$
$$\young({}{}{}{}{}\bullet,{}{}{}{},{}{}{}{},{}{},{}{})\,\,\,\,\,\,\,\,\,\,\,\,\,\, \young({}{}{}{}{}{},{}{}{}\bullet,{}{}{}{},{}{},{}{})\,\,\,\,\,\,\,\,\,\,\,\,\,\, \young({}{}{}{}{}{},{}{}{}{},{}{}{}{},{}\bullet,{}{})\,\,\,\,\,\,\,\,\,\,\,\,\,\, \young({}{}{}{}{}{},{}{}{}{},{}{}{}{},{}{},{}{},\bullet)$$
We will write $\mu\nearrow \lambda$ to say that the partition $\lambda$ can be obtained from the partition $\mu$ by adding only one square in an exterior corner of the Young diagram of $\mu.$ We extend this to bipartitions and we write $(\mu_1,\mu_2)\nearrow (\lambda_1,\lambda_2)$ if $\mu_1\nearrow \lambda_1$ and $\mu_2=\lambda_2$ or $\mu_1=\lambda_1$ and $\mu_2\nearrow \lambda_2.$ 
\begin{prop}
The induced representation $1_{\diag (\mathcal{H}_{n-1})}^{\mathcal{H}_n\times \mathcal{H}_{n-1}}$ is multiplicity free and its character is:
\begin{equation}
\character (1_{\diag (\mathcal{H}_{n-1})}^{\mathcal{H}_n\times \mathcal{H}_{n-1}})=\sum_{(\sigma_1,\sigma_2)\nearrow (\rho_1,\rho_2),\atop{ (\rho_1,\rho_2)\vdash n}}\chi^{(\rho_1,\rho_2)}\times \chi^{(\sigma_1,\sigma_2)}.
\end{equation}
\end{prop}
\begin{proof} By Corollary \ref{cor:sym_Gelf_pair}, $(\mathcal{H}_n\times\mathcal{H}_{n-1}, \diag (\mathcal{H}_{n-1}) )$ is a Gelfand pair. This implies that $1_{\diag (\mathcal{H}_{n-1})}^{\mathcal{H}_n\times \mathcal{H}_{n-1}}$ is multiplicity free by Theorem \ref{th}. Now, take any irreducible character of the product group $\mathcal{H}_n\times \mathcal{H}_{n-1}$ which should have the form $\chi^{(\lambda_1,\lambda_2)}\times \chi^{(\delta_1,\delta_2)}$ where $(\lambda_1,\lambda_2)\vdash n$ and $(\delta_1,\delta_2)\vdash n-1.$ By the Frobenius reciprocity theorem (see for example \cite[Theorem 1.12.6]{sagan2001symmetric}) we have:
\begin{equation}\label{hyp_frob}
\langle\ \chi^{(\lambda_1,\lambda_2)}\times \chi^{(\delta_1,\delta_2)},\character (1_{\diag (\mathcal{H}_{n-1})}^{\mathcal{H}_n\times \mathcal{H}_{n-1}})\rangle\ _{\mathcal{H}_n\times \mathcal{H}_{n-1}}
=\langle\ \Res\chi^{(\lambda_1,\lambda_2)}\times \chi^{(\delta_1,\delta_2)},1\rangle\ _{\diag(\mathcal{H}_{n-1})}
\end{equation}
\begin{equation*}
=\frac{1}{|\mathcal{H}_{n-1}|}\sum_{h\in \mathcal{H}_{n-1}}\Res_{\mathcal{H}_{n-1}}^{\mathcal{H}_n}\chi^{(\lambda_1,\lambda_2)}(h) \chi^{(\delta_1,\delta_2)}(h),
\end{equation*}
where $\Res_{\mathcal{H}_{n-1}}^{\mathcal{H}_n}\chi^{(\lambda_1,\lambda_2)}$ denotes the restriction of the character $\chi^{(\lambda_1,\lambda_2)}$ of $\mathcal{H}_n$ to $\mathcal{H}_{n-1}.$
But 
\begin{equation*}
\Res_{\mathcal{H}_{n-1}}^{\mathcal{H}_n}\chi^{(\lambda_1,\lambda_2)}=\sum_{(\sigma_1,\sigma_2)\nearrow (\lambda_1,\lambda_2)}\chi^{(\sigma_1,\sigma_2)},
\end{equation*}
as it is shown in \cite[Section 6.1.9]{geck2000characters}. Thus, using the orthogonality property of the characters of $\mathcal{H}_{n-1},$ Equation (\ref{hyp_frob}) becomes:
\begin{equation}
\langle\ \chi^{(\lambda_1,\lambda_2)}\times \chi^{(\delta_1,\delta_2)},\character (1_{\diag (\mathcal{H}_{n-1})}^{\mathcal{H}_n\times \mathcal{H}_{n-1}})\rangle\ _{\mathcal{H}_n\times \mathcal{H}_{n-1}}=\begin{cases}
1 & \text{ if $(\delta_1,\delta_2)\nearrow (\lambda_1,\lambda_2)$}\\
0 & \text{ otherwise.}
\end{cases}
\end{equation}
On the other hand, we have
\begin{equation}
\langle\ \chi^{(\lambda_1,\lambda_2)}\times \chi^{(\delta_1,\delta_2)},\sum_{(\sigma_1,\sigma_2)\nearrow (\rho_1,\rho_2),\atop{ (\rho_1,\rho_2)\vdash n}}\chi^{(\rho_1,\rho_2)}\times \chi^{(\sigma_1,\sigma_2)}\rangle\ _{\mathcal{H}_n\times \mathcal{H}_{n-1}}
\end{equation}
\begin{equation*}
=\sum_{(\sigma_1,\sigma_2)\nearrow (\rho_1,\rho_2),\atop{ (\rho_1,\rho_2)\vdash n}}\langle\ \chi^{(\lambda_1,\lambda_2)},\chi^{(\rho_1,\rho_2)}\rangle\ _{\mathcal{H}_n}\langle\  \chi^{(\delta_1,\delta_2)},\chi^{(\sigma_1,\sigma_2)}\rangle\ _{\mathcal{H}_{n-1}}
\end{equation*}
\begin{equation*}
=\begin{cases}
1 & \text{ if $(\delta_1,\delta_2)\nearrow (\lambda_1,\lambda_2)$}\\
0 & \text{ otherwise.}
\end{cases}
\end{equation*}
This ends the proof of this proposition.

\end{proof}

\begin{cor} The zonal spherical functions of the pair $(\mathcal{H}_n\times\mathcal{H}_{n-1}, \diag (\mathcal{H}_{n-1}) )$ are
\begin{equation*}
\omega^{(\sigma_1,\sigma_2)\nearrow (\rho_1,\rho_2)}(x,y)=\frac{1}{2^{n-1}(n-1)!}\sum_{h\in \mathcal{H}_{n-1}}\chi^{(\rho_1,\rho_2)}(xh)\chi^{(\sigma_1,\sigma_2)}(yh).
\end{equation*}
\end{cor}

The $\mathcal{H}_{n-1}$-generalised characters of $\mathcal{H}_n$ are given by 
\begin{equation}\label{gen_char_hyp}
\varpi^{(\sigma_1,\sigma_2)\nearrow (\rho_1,\rho_2)}(x)=\frac{\chi^{(\sigma_1,\sigma_2)}(1)}{2^{n-1}(n-1)!}\sum_{h\in \mathcal{H}_{n-1}}\chi^{(\rho_1,\rho_2)}(xh)\chi^{(\sigma_1,\sigma_2)}(h) \text{ for any $x\in \mathcal{H}_{n},$}
\end{equation}
where $(\sigma_1,\sigma_2)$ is a bipartition of $n-1$ and $(\rho_1,\rho_2)$ is a bipartition of $n$ with $(\sigma_1,\sigma_2)\nearrow (\rho_1,\rho_2).$ From the fact that 
\begin{equation*}
\chi^{(\sigma_1,\sigma_2)}(1)=(n-1)!\frac{\chi^{\sigma_1}(1)\chi^{\sigma_2}(1)}{|\sigma_1|!|\sigma_2|!},
\end{equation*}
the expression (\ref{gen_char_hyp}) can be reduced to 
\begin{equation*}
\varpi^{(\sigma_1,\sigma_2)\nearrow (\rho_1,\rho_2)}(x)=\frac{\chi^{\sigma_1}(1)\chi^{\sigma_2}(1)}{2^{n-1}|\sigma_1|!|\sigma_2|!}\sum_{h\in \mathcal{H}_{n-1}}\chi^{(\rho_1,\rho_2)}(xh)\chi^{(\sigma_1,\sigma_2)}(h) \text{ for any $x\in \mathcal{H}_{n}.$}
\end{equation*}
A Murnagham-Nakayama rule for the characters of the hyperoctahedral group appears in \cite{halverson1996murnaghan}. As a generalisation of the work of Strahov in \cite{strahov2007generalized}, it would be natural to investigate a Murnaghan–Nakayama type rule and other combinatorial identities for the $\mathcal{H}_{n-1}$-generalised characters of $\mathcal{H}_n.$

\bibliographystyle{plain}
% use the following instead if you encounter problems 
%\bibliographystyle{alpha}
\bibliography{biblio}

\end{document}